\documentclass[12pt]{amsart}
\usepackage{amsmath, amssymb, amsthm, bm, mathrsfs}
\usepackage{graphicx}
\usepackage{cite}

\newtheorem{theorem}{Theorem}[section]
\newtheorem{lemma}[theorem]{Lemma}
\newtheorem{corollary}[theorem]{Corollary}
\newtheorem{example}[theorem]{Example}
\newtheorem{remark}[theorem]{Remark}

\newcommand\bC{\mathbb{C}}
\newcommand\res{\operatornamewithlimits{res}}
\newcommand\myRe{\operatorname{Re}}
\newcommand\myIm{\operatorname{Im}}
\newcommand\dom{\operatorname{dom}}

\newcommand\ls{\operatorname{ls}}
\newcommand\sumI{\sideset{}{^{(1)}}\sum}
\newcommand\eps{\varepsilon}

\begin{document}

\title[Rank-one perturbations]
{Direct and inverse spectral problems for rank-one perturbations of self-adjoint operators}

\author[O.~Dobosevych]{Oles Dobosevych} 
\address{Ukrainian Catholic University, 2a Kozelnytska str, 79026 Lviv, Ukraine}
\email{dobosevych@ucu.edu.ua}

\author[R.~Hryniv]{Rostyslav Hryniv}
\address{Ukrainian Catholic University, 2a Kozelnytska str, 79026 Lviv, Ukraine \and University of Rzesz\'{o}w\\ 1\, Pigonia str.\\ 35-959 Rzesz\'{o}w, Poland}
\email{rhryniv@ucu.edu.ua, rhryniv@ur.edu.pl}

\subjclass[2010]{Primary: 47A55,  Secondary: 47A10, 15A18, 15A60}%
\keywords{Operators, rank-one perturbations, non-simple eigenvalues, eigenvalue asymptotics}%

\date{12 June 2020}%

\begin{abstract}
	For a given self-adjoint operator $A$ with discrete spectrum, we completely characterize possible eigenvalues of its rank-one perturbations~$B$ and discuss the inverse problem of reconstructing $B$ from its spectrum.
\end{abstract}

\maketitle

\section{Introduction}\label{sec:intro}

The main aim of this paper is to give a complete answer to the question, what spectra rank-one perturbations $B=A + \langle \cdot, \varphi \rangle \psi$ of a given self-adjoint operator~$A$ with simple discrete spectrum may have. There are several reasons why this question is of interest. Firstly, such perturbations lead to the explicit formulae of perturbation theory and thus many related questions can be fully answered. Secondly, despite its simplicity, the model offers extremely rich family of perturbed spectra. Namely, the main results of this paper show that, apart from the prescribed asymptotic distribution of eigenvalues, the spectrum of a rank-one perturbation~$B$ of~$A$ might become arbitrary---in particular, it may get eigenvalues of arbitrarily prescribed multiplicities in an arbitrarily prescribed finite set of complex points. In addition, we suggest an explicit method of constructing rank-one perturbations of~$A$ with a given admissible spectrum.  

Similar questions in finite-dimensional case have been studied since 1990-ies. In particular, Krupnik proved in~\cite{Kru92} that, given an arbitrary $n\times n$ matrix $A$, the spectrum of its rank-one perturbation can become an arbitrary complex $n$-tuple;  that was then further specified for the classes of Hermitian, unitary, and normal matrices. Savchenko~\cite{Sav03} studied the changes in the Jordan structure of~$A$ under a rank-one perturbation and found out that, generically, in each root subspace, only the longest Jordan chain splits. For low-rank perturbations, that result was further generalized in~\cite{Sav04} and independently in \cite{MorDop03}.  In~\cite{Far16}, the number of distinct eigenvalues of a matrix~$B$ was estimated in terms of some spectral characteristics of~$A$ and the rank of the perturbation. One should mention that earlier, H\"ormander and Melin~\cite{HorMel94} explained similar effects of rank-one perturbations in an infinite-dimensional setting; recently, Behrndt a.o.~\cite{BehLebPerMoeTru15} discussed possible changes to Jordan structure of an arbitrary linear operator~$A$ in a Banach space under general finite-rank perturbations.

For structured matrices and matrix pencils, a detailed rank-one perturbation theory and its application in the control theory was recently developed in a series of papers by Mehl a.o.~\cite{MehMehRanRod11, MehMehRanRod12, MehMehRanRod13, MehMehRanRod14, MehMehRanRod16, MehMehWoj17, SosMorMeh20}. The results established therein include e.g. changes in the Jordan structure of~$A$ under perturbation within classes of matrices enjoying certain real or complex Hamiltonian symmetry~\cite{MehMehRanRod11,MehMehRanRod13, MehMehRanRod16}, or for $H$-Hermitian matrices, with (skew-)Hermitian $H$, using the canonical form of the pair $(B, H)$~\cite{MehMehRanRod12, MehMehRanRod14}. Rank-one perturbations of matrix pencils and an important eigenvalue placement problem were studied e.g.\ in~\cite{MehMehWoj17, GerTru17, BarRoc20}, while a more general perturbation theory for structured matrices was outlined in the recent paper~\cite{SosMorMeh20}.

The cited results are mostly essentially finite-dimensional in the sense that their methods do not allow straightforward generalization to the infinite-dimensional case (see, however, \cite{HorMel94, BehMoeTru14}). The latter has been studied within the general spectral theory for bounded or unbounded operators in infinite-dimensional Banach or Hilbert spaces~\cite{Kat95}. For instance, a comprehensive spectral analysis of rank-one perturbations of unbounded self-adjoint operators is carried out in~\cite{Sim95}, where a detailed characterization of discrete, absolutely continuous, and singlularly continuous spectra of the perturbation~$B$ is given. A thorough overview of the theory of Schr\"odinger operators under singular point perturbations (formally corresponding to additive Dirac delta-functions and their derivatives) is given in the monographs by Albeverio a.o.~\cite{AGHH, AlbKur00}, suggesting also comprehensive reference lists. Much attention has been paid to the so-called singular and super-singular rank-one or finite-rank perturbations of self-adjoint operators, where the functions $\varphi$ and $\psi$ belong to the scales of Hilbert spaces $\operatorname{dom}(A^\alpha)$ with negative~$\alpha$, see e.g.~\cite{AlbKonKos05, AlbKosKurNiz03, AlbKos99, AlbKuzNiz08,  DudVdo16, Gol18, Kur04, KurLugNeu19, KuzNiz06, AlbKur97, AlbKur97a, AlbKur99}; in this case, a typical approach is through the Krein extension theory of self-adjoint operators. Rank-one and finite-rank  perturbations of self-adjoint operators in Krein spaces have been recently discussed in e.g.~\cite{BehMoeTru14, BehLebPerMoeTru16}.

Despite the extensive research in the area, there seems to be no complete understanding what spectra rank-one perturbations of a given opearator $A$ can produce. As the earlier research demonstrates (cf.~\cite{Sim95,AGHH,AlbKur00}), the question is quite non-trivial even for self-adjoint perturbations of a self-adjoint operator~$A$, and thus necessarily much more complicated for generic rank-one perturbations. In our previous work~\cite{DobHry20}, we described local spectral properties of rank-one perturbations of a self-adjoint operator with discrete spectrum. Namely, it was shown therein that, as in the finite-dimensional case~\cite{Kru92}, such a perturbation can possess eigenvalues of arbitrarily prescribed multiplicities at any finite set of complex numbers; cf.~\cite{HomHry20} for similar results for the class of $\mathcal{PT}$-symmetric perturbations. 

The main aim of the present paper is to give a complete description of the possible spectra of rank-one perturbations of a given self-adjoint operator~$A$ with simple discrete spectrum. More exactly, with $\lambda_n$ denoting the eigenvalues of $A$, Theorem~\ref{thm:EV-asympt} states that the eigenvalues of a rank-one perturbation~$B$ can be labelled as $\mu_n$ (counting with multiplicities) so that the sum of all offsets $|\mu_n-\lambda_n|$ is finite. Moreover, it turns out that every sequence $\mu_n$ with this property can be a spectrum for such a $B$; Theorem~\ref{thm:inv} in addition suggests a method for constructing all such rank-one perturbations~$B$.

The structure of the paper is as follows. In the next section, we collect some basic spectral properties of the rank-one perturbations~$B$. In Section~\ref{sec:asymptotics}, the asymptotic distribution of eigenvalues of $B$ is studied and, in particular, summability of the offsets $|\mu_n-\lambda_n|$ is proved. Sufficiency of this condition, as well as an algorithm for constructing a rank-one perturbation~$B$ with a prescribed admissible spectrum are established in Section~\ref{sec:inverse}. Finally, in Section~\ref{sec:example}, we give two examples and discuss straightforward generalizations of the main results to wider classes of the operators~$A$.



\section{Preliminaries}\label{sec:general}


In this section, we collect some properties of the rank-one perturbations of self-adjoint operators~$A$ acting in a fixed fixed separable (infinite-dimensional) Hilbert space~$H$ that will be required to prove the main results of this work. Throughout the whole paper, we shall assume that 
\begin{itemize}
  \item[(A1)] the operator~$A$ is self-adjoint and has simple discrete spectrum.
\end{itemize}
The operator~$A$ is necessarily unbounded but it may be bounded below or above; without loss of generality, in this case we assume that $A$ is bounded below (otherwise, we just replace $A$ with $-A$). Under these assumptions, the spectrum of~$A$ consists of simple real eigenvalues that can be listed in increasing order as $\lambda_n$, $n\in I$, with~$I = \mathbb{N}$ if $A$ is bounded below and $I=\mathbb{Z}$ otherwise. Keeping in mind the most important and interesting applications to the differential operators, we make an additional assumption that
\begin{itemize}
  \item[(A2)] the eigenvalues of~$A$ are separated, i.e.,
  \begin{equation}\label{eq:dist}
  \inf_{n \in I} |\lambda_{n+1} - \lambda_n| =: d > 0.
  \end{equation}
\end{itemize}

The operator~$B$ is a rank-one perturbation of the operator~$A$, i.e.,
\begin{equation}\label{eq:B}
    B=A + \langle \cdot, \varphi \rangle \psi
\end{equation}
with fixed non-zero vectors~$\varphi$ and $\psi$ in~$H$ and with~$\langle\,\cdot\,,\,\cdot\,\rangle$ denoting a scalar product in~$H$. Clearly, $B$
is well defined and closed on its natural domain $\dom(B)$ equal to $\dom(A)$.
Next, for $\lambda \in \rho(A)$, we introduce the \emph{characteristic function}
\begin{equation}\label{eq:F}
    F(\lambda) := \langle (A-\lambda)^{-1}\psi, \varphi \rangle + 1
\end{equation}
and denote by $\mathcal{N}_F$ the set of zeros of $F$. This function appears in the Krein resolvent formula for~$B$~\cite{AlbKur00, DobHry20}, and its zeros characterise the spectrum of~$B$.

To be more specific, we denote by $v_n$ a normalized eigenvector of~$A$ corresponding to the eigenvalue~$\lambda_n$; then the set $\{v_n\}_{n\in I}$ forms an orthonormal basis of~$H$, and we let $a_n$ and $b_n$ be the corresponding Fourier coefficients of the vectors~$\varphi$ and $\psi$, so that 
\[
	\varphi = \sum_{k\in I} a_k v_k, \qquad \psi = \sum_{k\in I} b_k v_k.
\]
Now we set 
\[
	I_0 := \{n \in I \mid a_nb_n =0\}, \qquad 	I_1 := \{n \in I \mid a_nb_n \ne 0\}
\]
and $\sigma_j(A) := \{\lambda_n \mid n\in I_j\}$; then $\sigma_0(A) = \sigma_0(B) := \sigma(A) \cap \sigma(B)$ is the common part of the spectra of~$A$ and $B$, while the spectrum of~$B$ in~$\bC \setminus \sigma_0(A)$ coincides with the set of zeros of~$F$. 

In fact, the function~$F$ also characterizes eigenvalue multiplicities of the operator~$B$. We recall that the \emph{geometric} multiplicity of an eigen\-value~$\lambda$ of~$B$ is the dimension of the null-space of the operator~$B-\lambda$, while its \emph{algebraic} multiplicity is the dimension of the corresponding root subspace, i.e., of the set of all $y \in \dom (B)$ such that $(B-\lambda)^k y =0$ for some $k\in\mathbb{N}$. Next, by the spectral theorem for~$A$, the characteristic function~$F$ of~\eqref{eq:F} can be written as%
\begin{footnote}
	{In what follows, the summations and products over the index sets that are not bounded from below or above are understood in the principal value sense}
\end{footnote}
\begin{equation}\label{eq:F-new}
	F(z) =  \sum_{k\in I_1} \frac{\overline{a_k}b_k}{\lambda_k - z} + 1 
\end{equation}
and thus can be analytically extended to $\sigma_0(A)$; we keep the notation~$F$ for this extension. 

As proved in~\cite{DobHry20}, the geometric multiplicity of every eigenvalue~$\mu$ of~$B$ is at most~$2$; multiplicity~$2$ is only possible when~$\mu \in \sigma_0(A)$, i.e., $\mu = \lambda_n$ for some $n\in I_0$ and, in addition, $a_n = b_n = F(\lambda_n) = 0$. We also observe that when $a_n = b_n =0$, then the subspace $\ls\{v_n\}$ is invariant under both~$B$ and $B^*$ and thus is reducing for~$B$. Denoting by~$H_0$ the closed linear span of all such subspaces, we conclude that $H_0$ and $H \ominus H_0$ are reducing for~$B$ and the operators~$A$ and $B$ coincide on~$H_0$. For that reason, only the part of $B$ in $H \ominus H_0$ is of interest, and, without loss of generality, we may assume that $H_0 = \{0\}$.

Under such an assumption, every eigenvalue $\mu$ of~$B$ is geometrically simple. One of the main results of~\cite{DobHry20} claims that the algebraic multiplicity~$m$ of an eigenvalue $\mu$ of $B$ coincides with the multiplicity $l$ of~$z = \mu$ as a zero of~$F$ if $\mu \not\in \sigma_0(A)$ and is equal to $l+1$ otherwise. In other words, the common spectrum $\sigma_0(A) = \sigma_0(B)$ of $A$ and $B$ and the zeros of $F$ completely characterize the spectrum of $B$, counting with multiplicities. This allows us to reduce the study of the eigenvalue distribution for the perturbation~$B$ to the study of zero distribution of the characteristic function~$F$. 



\section{Eigenvalue distribution of the operator $B$}\label{sec:asymptotics}

In this section, we shall discuss eigenvalue distribution of the rank-one perturbation~$B$ of~$A$ given by~\eqref{eq:B}. 
As explained in the previous section, the spectrum of $B$ consists of two parts: $\sigma_0(B) = \sigma(A) \cap \sigma(B)$, which is the common part of the spectra of $A$ and $B$, and $\sigma_1(B)$, which is the set of zeros of the characteristic function
\[
    F(z) = \sum_{n \in I_1} \frac {\overline {a_n} b_n} {\lambda_n - z} + 1
\]
in the domain $\mathbb{C}\setminus\sigma_1(A)$; moreover, the algebraic multiplicity of an eigenvalue~$\mu \in \sigma(B)$ is determined by its multiplicity as a zero of the characteristic function~$F$.

The main result of this section is given by the following theorem.

\begin{theorem}\label{thm:EV-asympt}
The eigenvalues of the operator $B$ can be labelled as $\mu_n$, $n \in I$, in such a way that the series
\begin{equation}\label{eq:mu-converge}
    \sum_{n \in I} | \mu_n - \lambda_n|
\end{equation}
converges. In particular, all but finitely many eigenvalues of~$B$ are simple.
\end{theorem}

First we shall show that large enough elements of~$\sigma_1(B)$ are located near $\sigma_1(A)$, which will enable their proper enumeration. To begin with, for $k \in I_1$ we define the functions $G_k$ and $H_k$ by the formulas%
\begin{footnote}
{Throughout the paper, the symbol $\sum{\hspace*{-2pt}\vphantom{\sum}}^{(1)}$ will denote summation over the index set $I_1$}
\end{footnote}
\[
    G_k(z) = \frac{\overline {a_k} b_k} {\lambda_k - z} + 1,
    \qquad
    H_k(z) = \sumI_{|n| \le k} \frac {\overline {a_n} b_n} {\lambda_n - z} + 1
\]
and 
introduce the sets
\begin{gather*}
  Q_k
    {:=} \{z \in \mathbb{C} \, \mid \,
		\myRe(z),\myIm(z) \in [\lambda_{-|k|} - \tfrac{d}2, \lambda_{|k|} + \tfrac{d}2]
            \},\\
	 R_k {:=} \{z \in \mathbb {C} \, \mid \, |z - \lambda_k|  < \tfrac{d}2 \},
\end{gather*}
where we replace $\lambda_{-|k|}$ with $-\lambda_{|k|}$ if $I=\mathbb{N}$.
Due to the assumption~(A2) the sets $R_k$ are pairwise disjoint and also $R_k \cap Q_n = \emptyset$ if $|k|>|n|$.

\begin{lemma}\label{lem:Keps}
For every $\eps >0$ there exist integers $K_\eps>0$ and $K_\eps'>K_\eps$ such that the following holds:
\begin{itemize}
  \item[(a)] for every $k$ with $|k|> K_\eps$ and every $z \in \overline{R_k} = \partial R_k \cup R_k$
    \begin{equation}\label{es:6}
        \sumI_{\substack{|n| >  K_\eps \\ n \ne k}} \left | \frac {\overline {a_n} b_n} {\lambda_n - z} \right | < \frac{2\eps} {d};
    \end{equation}
  \item[(b)] for every $z \in \mathbb{C}\setminus Q_{K_\eps'}$
    \begin{equation}\label{es:5}
        \sumI_{|n| \le K_\eps} \left | \frac{\overline {a_n} b_n}{\lambda_n - z} \right | < \eps.
    \end{equation}   
\end{itemize}
\end{lemma}

\begin{proof}
The sequences $(a_n)_{n\in I}$ and $(b_n)_{n\in I}$ of the Fourier coefficients of the vectors $\phi$ and $\psi$ are square summable, so that, by the Cauchy--Bunyakowsky--Schwarz inequality,
\[
    \sum_{n \in I_1} |{\overline {a_n} b_n}| < \infty.
\]
Therefore, for every $\eps>0$ there exists a $K_\eps$ such that
\[
    \sumI_{|n| > K_\eps} |{\overline {a_n} b_n}|< \eps.
\]
Take a $k$ satisfying $|k|> K_\eps$; then by virtue of Assumption~(A2) for every $z \in \overline{R_k}$ and every $n\ne k$ we get $|\lambda_n - z| \ge \tfrac{d}2$, and therefore \eqref{es:6} holds.

For part (b), note that $|\lambda_n - z| > (K'_\eps- K_\eps)d$ if $|n| \le K_\eps$ and $z \in \mathbb{C}\setminus Q_k$ with $|k|\ge K'_\eps>K_\eps$; therefore, by choosing $K'_\eps$ large enough, we arrive at~\eqref{es:5}.
\end{proof}

\begin{corollary}\label{cor:zeros-local}
	Take $\eps := d/(2+d)$ and $K'_\eps$ as in the above lemma; then 
	\[
		\sigma(B) \subset Q_{K'_\eps} \cup \Bigl(\bigcup\nolimits_{n\in I} R_n \Bigr).
	\]
	
	Indeed, it suffices to note that if $z$ is outside $Q_{K'_\eps}$ and every $R_n$, $n\in I$, then $|\lambda_n - z|\ge d/2$, so that 
	\[
		\sumI_{\substack{|n| >  K_\eps}} \left | \frac {\overline {a_n} b_n} {\lambda_n - z} \right | < \frac{2\eps} {d},
	\]
	which together with part (b) of that lemma shows that 
	\[
		|F(z)| \ge 1 - \sumI_{\substack{|n| >  K_\eps}} \left | \frac {\overline {a_n} b_n} {\lambda_n - z} \right | 
			> 1 - \eps (1 + 2/d) = 0.
	\]
\end{corollary}

\begin{lemma}\label{lem:zeros-local}
There exists $K>0$ such that for all $k\in I_1$ with $|k| > K$ the following holds:
\begin{itemize}
  \item[(a)] the function $F$ has exactly one zero in $R_k$;
  \item[(b)] the functions $H_k$ and $F$ have the same number of zeros in $Q_k$.
\end{itemize}
\end{lemma}

\begin{proof}
Fix an $\eps\in(0,d/2)$ such that
\[
    \eps \Bigl(1 + \frac4d \Bigr) <1;
\]
we shall show that (a) and (b) hold for $K = K'_\eps$ of Lemma~\ref{lem:Keps}. 

If $k$ satisfies $|k| > K$, then by Lemma~\ref{lem:Keps} for every $z \in \partial R_k$ we get
\begin{equation*}\label{es:7}
    |F(z) - G_k(z)|
         \le \sumI_{|n| \le K_\eps} \left | \frac{\overline{a_n} b_n}{\lambda_n - z} \right|
        + \sumI_{\substack{|n| >  K_\eps\\ n \neq k}}
            \left| \frac {\overline{a_n} b_n}{\lambda_n - z} \right|
         < \eps + \frac{2\eps} {d}.
\end{equation*}
On the other hand, $|\overline{a_k} b_k| < \eps$ if $k\in I_1$ satisfies $|k| > K> K_\eps$, and then
\begin{equation}\label{es:8}
    |G_k(z)|
        \ge 1 - \left| \frac{\overline{a_k} b_k}{\lambda - \lambda_k} \right|
        > 1 -  \frac{2\eps}{d}
\end{equation}
for all $z \in \partial R_k$. By the choice of $\eps$ we conclude that then
\begin{equation}\label{es:9}
|G_k(z)| > |F(z) - G_k(z)|
\end{equation}
for all such $z$. As the functions $G_k$ and $F$ both have  the same number of poles in $R_k$ (namely, a simple pole at $\lambda_k$), by estimate \eqref{es:9} and Rouche's theorem they have the same number of zeros in the set $R_k$. By virtue of inequality \eqref{es:8}, the unique zero $z = \lambda_k + \overline{a_k} b_k$ of the function $G_k$ belongs to the circle $R_k$ for all $k\in I_1$ with $|k|> K$, and thus the function $F$ has exactly one zero in $R_k$ for such $k$ as well. This completes the proof of part~(a).

Next, by the definition of the set $Q_k$, it holds that
$|\lambda_n - z| \ge \tfrac{d}2$ if $z \in \partial Q_k$ and $|n|> |k|$. 
By the choice of the number $K_\eps$, we find that
\begin{equation*}\label{es:13}
    |F(z) - H_k(z)|
         \le \sumI_{|n| > |k|} \left|  \frac {\overline {a_n} b_n} {\lambda_n - z} \right|
         < \frac {2\eps} {d}
\end{equation*}
and 
\begin{equation}\label{es:12}
    \sumI_{K_\eps < |n| \le |k|}
        \left| \frac{\overline {a_n} b_n}{\lambda_n - z} \right|
        < \frac{2\eps}{d}
\end{equation}
if $|k|> K_\eps$ and $z \in \partial Q_k$. Also, by part~(b) of Lemma~\ref{lem:Keps} we have
\begin{equation}\label{es:10} 
   \sumI_{|n|\le K_\eps} \left | \frac{\overline {a_n} b_n}{\lambda_n - z} \right | < \eps
\end{equation}
as soon as $|k| > K$ and $z \in \mathbb{C}\setminus Q_k$.
Combining estimates \eqref{es:12} and \eqref{es:10}, we conclude that
\begin{equation}\label{es:14}
    |H_k(z)|  \ge 1 - \sumI_{|n| \le k}
        \left |\frac{\overline {a_n} b_n}{\lambda_n - z} \right|
        > 1 - \eps - \frac{2\eps}{d}
\end{equation}
for all $k$ with $|k| > K$ and all $z \in \mathbb {C} \setminus Q_k$.

It follows that for $k$ with $|k| > K$ and for all $z \in \partial Q_k$
\begin{equation*}\label{es:15}
|H_k(z)| > |F(z) - H_k(z)|.
\end{equation*}
Since the functions $H_k$ and $F$ have the same poles in $Q_k$ (namely, simple poles $\lambda_n$ for $n\in I_1$ with $|n|\le k$), we conclude by Rouche's theorem that they have the same number of zeros in $Q_k$ for all $k>K$. The proof is complete.
\end{proof}

\begin{remark}
Take $k$ larger than $K$ of the above lemma and denote by $N_k$ the cardinality of the set $\sigma_1(A)\cap Q_k$. The function~$H_k$ is a ratio of two polynomials of degree $N_k$ and due to~\eqref{es:14} all its zeros are in $Q_k$. Therefore, the function~$F$ has precisely~$N_k$ zeros in~$Q_k$ counting with multiplicities.
\end{remark}

\begin{corollary}\label{cor:F-loc}
The zeros of $F$ in $\mathbb{C}\setminus \sigma_0(A)$ can be labelled (counting with multiplicities) as $\mu_k$ with $k\in I_1$ in such a way that $|\mu_k - \lambda_k| < \tfrac{d}2$ for all $k\in I_1$ with $|k|>K$.
\end{corollary}

Recalling the results of the previous section on relation between the eigenvalues of $B$ and zeros of the function~$F$ in $\mathbb{C}\setminus\sigma_1(A)$, we arrive at the following conclusion.

\begin{corollary}\label{cor:mu-labeling}
  Eigenvalues of the operator~$B$ can be labelled (counting with mulitplicities) as $\mu_k$ with $k\in I$ in such a way that $|\mu_k - \lambda_k|< \tfrac{d}2$ when $|k|>K$, $K$ being the constant of Lemma~\ref{lem:zeros-local}.
\end{corollary}

\begin{proof}[Proof of Theorem~\ref{thm:EV-asympt}]
We fix an enumeration of $\mu_k$ as in Corollary~\ref{cor:mu-labeling}. Then $\mu_k = \lambda_k$ for all $k\in I_0$ with sufficiently large $|k|$, whence it suffices to prove that the series
\[
    \sum_{n \in I_1} |\mu_n - \lambda_n|
\]
is convergent. 

We take $\eps$ and $K$ as in Lemma~\ref{lem:zeros-local}; then, according to Corollary~\ref{cor:mu-labeling}, for every $k\in I_1$ with $|k|>K$ the eigenvalue $\mu_k \in R_k$ is a zero of~$F$, so that
$$
    F(\mu_k) = \sumI_{|n| \le K_\eps} \frac{\overline{a_n} b_n}{\mu_k - \lambda_n}
        + \sumI_{\substack{|n| >  K_\eps\\ n \neq k}} \frac {\overline{a_n} b_n} {\mu_k - \lambda_n}
        + \frac {\overline{a_k} b_k} {\mu_k - \lambda_k} +1 = 0
$$
and 
\[
    \left|\frac {\overline{a_k} b_k} {\mu_k - \lambda_k} \right| 
        > 1 - \sumI_{|n| \le K_\eps} \left|
        \frac{\overline{a_n} b_n}{\lambda_n - \mu_k} \right|
            - \sumI_{\substack{|n| >  K_\eps\\ n \neq k}}
        \left|\frac{\overline{a_n} b_n}{\lambda_n - \mu_k} \right|.
\]
By virtue of Lemma~\ref{lem:zeros-local} we conclude that 
$$
    \left|\frac {\overline{a_k} b_k} {\mu_k - \lambda_k} \right | 
        > 1 - \eps - \frac {2\eps} {d}
$$
for $k \in I_1$ with $|k|>K$. 
 Since $1 - \eps - \frac {2\eps} {d} > \frac {2\eps} {d}$, we find that 
\begin{equation}\label{eq:19}
    |\mu_k - \lambda_k| 
        < \frac{d}{2\eps}|\overline {a_k} b_k|
\end{equation}
for $k\in I_1$ with $|k| > K$. As the series $\sum_{n \in I_1} |\overline{a}_n b_n|$ is convergent, the same is true of the series $\sum_{n \in I_1} |\mu_n - \lambda_n|$, and the proof is complete.
\end {proof}


\section{Inverse spectral problem}\label{sec:inverse}

The purpose of this section is to study the inverse spectral problem, namely, the problem of reconstructing the operator~$B$ from its spectrum~$(\mu_n)_{n\in I}$ assuming that the operator $A$ is known.

More generally, let the operator~$A$ satisfy assumptions (A1) and (A2), i.e., is self-adjoint and has a simple discrete spectrum~$(\lambda_n)_{n\in I}$ that is $d$-separated as in~\eqref{eq:dist}.  Our aim is to find necessary and sufficient conditions that another sequence $(\nu_n)_{n\in I} $ of complex numbers must satisfy so that it could be a spectrum (counting with multiplicities) of an operator~$B$ of the form \eqref {eq:B}. Also, we want to suggest an algorithm of constructing the operator~$ B $ (i.e., the function $ \varphi $ and $\psi$) and investigate the uniqueness question.

The latter question can be answered straight ahead. Indeed, if the inverse problem for a sequence $(\nu_n)_{n \in I}$ has a solution, then it has many solutions. In fact, if
$$
	B_j = A + \langle \cdot, \varphi_j \rangle \psi_j, \quad j = 1,2
$$
and vectors $\varphi_1$, $\varphi_2 $, $ \psi_1 $, $ \psi_2 $ satisfy
$$
	\overline {\langle \varphi_1, v_n \rangle} \langle \psi_1, v_n \rangle
	= \overline {\langle \varphi_2, v_n \rangle} \langle \psi_2, v_n \rangle, \quad n \in \mathbb {N},
$$
then the spectra of $B_1$ and $B_2$ coincide counting with multiplicities. Therefore, in the inverse problem one can only restore the products $\overline{a}_nb_n$ of the Fourier coefficients of the functions $ \varphi $ and $ \psi $, which are the residues of the function $ -F $ of~\eqref{eq:F-new}.

The main result of this section is given by the following theorem.

\begin{theorem}\label{thm:inv}
Assume that a sequence $\bm{\nu}$ of complex numbers can be enumerated as $\nu_n$, $n\in I$, in such a way that the series 
\begin{equation}\label{eq:nu-converge}
	\sum_{n \in I} |\nu_n - \lambda_n| 
\end{equation}
converges. Then there exist vectors $\varphi,\psi\in H$ such that the spectrum of~$B$ coincides with $\bm{\nu}$ counting with multiplicities.
\end{theorem}

Let us denote by $I_0$ the set of indices $n\in I$ for which $\lambda_n$ appears in~$\bm{\nu}$ and set $\Lambda_0 :=\{\lambda_n \mid n\in I_0\}$. Convergence of the series~\eqref{eq:nu-converge} implies that for every $\varepsilon\in(0,d/2)$ there exists a $K>0$ such that $|\nu_n - \lambda_n| < \varepsilon$ for all $n\in I$ with $|n|>K$. Therefore, if $n\in I_0$ and $|n|>K$, then $\nu_n = \lambda_n$, and without loss of generality we may assume that $\nu_n = \lambda_n$ for all $n\in I_0$.

We also set $I_1:= I \setminus I_0$, $\Lambda_1:=\{\lambda_n \mid n\in I_1\}$, and introduce the function
\begin{equation}\label{eq:prod-conv}
	\tilde{F}(z):= \prod_{n\in I_1} \frac{\nu_n - z}{\lambda_n - z}.
\end{equation}
To show that $\tilde F$ is well defined, we take an arbitrary $\varepsilon \in (0,d/2)$ and set 
\begin{equation*}		
	R_n(\varepsilon) := \{ z \in \bC \mid |z-\lambda_n| < \varepsilon\}, \qquad
	R(\varepsilon) := \bC \setminus \bigl(\cup_{n\in I_1} R_n(\varepsilon)\bigr).
\end{equation*}
Then we have the following

\begin{lemma}\label{lem:prod-conv}
For each $\varepsilon \in (0, d/2)$, the product in~\eqref{eq:prod-conv} 
converges uniformly in $R(\varepsilon)$.
\end{lemma}

\begin{proof}
It is enough to show that the series
\begin{equation*}\label{eq:prod-conv-1}
	\sum_{n\in I_1} \log \left(1 + \left| \frac {\nu_n - \lambda_n}{\lambda_n - z}\right| \right)
\end{equation*}
converges uniformly on the same set. However, for $z \in R(\varepsilon)$ and $n\in I_1$ we get the estimate
\begin{equation}\label{eq:prod-conv-2}
	\log\left(1 + \left| \frac {\nu_n - \lambda_n}{\lambda_n - z}\right| \right)
		\le \left| \frac {\nu_n - \lambda_n}{\lambda_n - z}\right|
		\le \frac {|\nu_n - \lambda_n|}{\varepsilon},
\end{equation}
which in view of the convergence of the series~\eqref{eq:nu-converge} and the Weierstrass M-test finishes the proof.
\end{proof}

The Weierstrass $M$-test used in the above proof also justifies passage to the limit 
\[
	\lim_{u \to +\infty} \sum_{n\in I_1} \left| \frac {\nu_n - \lambda_n}{\lambda_n - iu}\right| = 0;
\]
as a result, we get 

\begin{corollary}\label{cor:lim-t-F}
There exists the limit
$$
\lim_{u \to + \infty} \tilde{F} (i u) = 1.
$$
\end{corollary}

The function $\tilde F$ is meromorphic in $\bC$, and its residue at the point $\lambda_n \in \Lambda_1$ is
\begin{equation}\label{eq:res-t-F}
	-c_n = \lim\limits_{z \to \lambda_n} (z-\lambda_n) \tilde{F}(z) 
	    = (\lambda_n - \nu_n) \prod\limits_{\substack{m \in I_1 \\ m \neq n}} 
	    \frac{\nu_m - \lambda_n}{\lambda_m - \lambda_n}.
\end{equation}

\begin{lemma}\label{lem:res-sum}
The series
\begin{equation}\label{eq:res-sum-1}
\sum_ {n \in I_1} |c_n|
\end{equation}
converges.
\end{lemma}

\begin {proof}
In view of~\eqref{eq:res-t-F}, convergence of series \eqref{eq:res-sum-1} follows from convergence of the series
\begin{equation*}\label{eq:res-sum-2}
	\sum_ {n \in I_1} | \lambda_n - \nu_n | 
			\prod_ {\substack{m \in I_1 \\ m \neq n}} 
		\left | \frac {\nu_m - \lambda_n} {\lambda_m - \lambda_n} \right |, 
\end{equation*}
and to establish the latter it is enough to show that the sequence
\begin{equation}\label{eq:res-sum-3}
\prod_{\substack{m \in I_1 \\ m \neq n}} \left | \frac {\nu_m - \lambda_n} {\lambda_m - \lambda_n} \right |
\end{equation}
is bounded in $n\in I_1$.

Applying the same reasoning as in the proof of Lemma~\ref{lem:prod-conv}, we conclude that the sum of the series
\begin{align*}
	\sumI_{m \neq n} \log \left | \frac {\nu_m - \lambda_n} {\lambda_m - \lambda_n} \right | 
		&\le \sumI_{m \neq n} \log 
			\left (1 + \left | \frac {\nu_m - \lambda _m} {\lambda_m - \lambda_n} \right | \right)\\ 
		&\le \sumI_{m \neq n} \left | \frac {\nu_m - \lambda _m} {\lambda_m - \lambda_n} \right |
			\le \frac1d \sum_ {m \in I_1} |\nu_m - \lambda _m|
\end{align*}
has an $n$-independent bound, which implies that the sequence~\eqref{eq:res-sum-3} is uniformly bounded. The proof is complete.
\end{proof}

In view of the above lemma, the series 
\[
	\sum_ {n \in I_1} \frac {c_n} {\lambda_n- z}
\]
converges uniformly in $R(\varepsilon)$ for every $\varepsilon \in (0,d/2)$. It follows that the function
\[
	F(z) :=  1 + \sum_ {n \in I_1} \frac {c_n} {\lambda_n- z}
\]
is well defined and analytic in the set $\bC \setminus \Lambda_1$ and has simple  poles at the points $z \in \Lambda_1$.  
The Lebesgue dominated convergence theorem also implies that 
$$
	\lim_{u \to +\infty} F(i u) = 1.
$$

\begin{lemma}\label{lem:F=tildeF}
	The function $F-\tilde F$ is equal to zero identically in~$\bC$.
\end{lemma}

\begin{proof}
We set $G := F-\tilde F$; then the function $G$ is meromorphic in $\bC$ with possible single poles at the points $\Lambda_1$. However, as the residua of $F$ and $\tilde F$ at each point $z\in \Lambda_1$ coincide by construction, we conclude that the function~$G$ has removable singularities at the points $z\in \Lambda_1$ and thus is entire. We next show that $G$ is uniformly bounded over $\bC$ and thus is constant by the Liouville theorem; as  
\[
	\lim_{u\to+\infty} G(iu) = \lim_{u\to+\infty} F(iu) - \lim_{u\to+\infty} \tilde F(iu) =0,
\]
this constant is zero, and thus the proof will be complete.
	
For large enough $k$, we denote by $Q_k$ the rectangular bounded by the lines $\myIm z = \pm \lambda_k$, $\myRe z = \lambda_{-|k|} - d/2$, and $\myRe z = \lambda_k + d/2$ (if $I=\mathbb{N}$, then we replace $\lambda_{-k}$ with $-\lambda_k$). 
Observe that for every $n \in I_1$ and $z \in \partial Q_k$ we have 
\(|z-\lambda_n|\ge d/2\); 
as a result, we conclude that 
\[
	\sup_{z\in\partial Q_k}|F(z)| \le 1 + \frac{2}{d} \sum_{n\in I_1} |c_n| := C.
\]
Next, we note that for $\eps\in(0,d/2)$ the boundary~$\partial Q_k$ of $Q_k$ lies in the set~$R(\eps)$. As in the proof of Lemma~\ref{lem:prod-conv}, we can derive the bound (cf.~\eqref{eq:prod-conv-2})
\[
	\sup_{z\in\partial Q_k}|\tilde F(z)| \le \exp\Bigl\{ \frac1\eps \sum_{n \in I_1} |\nu_n - \lambda_n|\Bigr\} := \tilde C.
\]
Since the function~$G$ is entire, it follows from the maximum modulus principle that 
\[
	|G(z)| \le C + \tilde C
\]
inside every set~$Q_n$ and thus for all $z\in \bC$. Therefore, the function $G$ is bounded; as explained at the beginning of the proof, this implies the required results. 
\end{proof}

\begin{proof}[Proof of Theorem~\ref{thm:inv}]
Given any sequence $\bm{\nu}$ of complex numbers satisfying the assumption of the theorem, we construct the meromorphic function~$\tilde{F}$ via~\eqref{eq:prod-conv}. Next, calculate the residua $-c_n$ of~$\tilde F$ at the points $\lambda_n \in \Lambda_1$ via~\eqref{eq:res-t-F} and define the sequences 
\begin{equation}\label{eq:anbn}
	a_n:= \sqrt{|c_n|}, \qquad b_n:= \sqrt{|c_n|} e^{i\arg{c_n}}, \qquad n\in I_1,
\end{equation}
and 
\[
	a_n:= 1/(1+|n|) ,\qquad b_n = 0, \qquad n\in I_0.
\]

Since the sequence $(c_n)_{n\in I_1}$ is summable by Lemma~\ref{lem:res-sum}, it follows that the sequences $(a_n)_{n\in I}$ and $(b_n)_{n\in I}$ belong to $\ell_2(I)$. Therefore, there exist functions $\varphi$ and $\psi$ in the Hilbert space $\mathcal{H}$ whose Fourier coefficients in the basis $v_n$ are equal to $a_n$ and $b_n$, respectively. 

We now consider the operator~$B$ of the form~\eqref{eq:B} with the functions $\varphi$ and $\psi$ just introduced and conclude by virtue of Lemma~\ref{lem:F=tildeF} that the corresponding meromorphic function~$F$ of~\eqref{eq:F-new} coincides with~$\tilde F$. Therefore, zeros of $F$ are precisely the elements of the subsequence~$\bm{\nu}_1:=(\nu_n)_{n\in I_1}$, both counting multiplicity; namely, if a number $\nu$ occurs $k$ times  in~$\bm{\nu}_1$, it is a zero of $F$ of multiplicity~$k$. The analysis of the paper~\cite{DobHry20} summarized in Section~\ref{sec:general} shows that each element $\nu$ of $\bm{\nu}$ is an eigenvalue of~$B$ and its multiplicity is equal to the number of times $\nu$ is repeated in the sequence~$\bm{\nu}$. The proof is complete.
\end{proof}

The above proof also gives an algorithm of constructing an operator~$B$ for any sequence~$\bm{\nu}$ of complex numbers satisfying~\eqref{eq:nu-converge}. Namely, given such a sequence~$\bm{\nu}$, we 
\begin{enumerate}
	\item construct the product~$\tilde F$ of~\eqref{eq:prod-conv};
	\item then calculate the residua~$-c_n$ of $\tilde F$ at the points~$\lambda_n$;
	\item construct the Fourier coefficients~$a_n$ and $b_n$ of $\varphi$ and $\psi$ via~\eqref{eq:anbn}.
\end{enumerate}
As was noted at the beginning of this section, there are infinitely many such operators; all of them are fixed by the condition~$\overline{a}_nb_n = c_n$ on the Fourier coefficient $a_n$ and $b_n$ of the functions $\varphi$ and $\psi$.


\section{Examples and discussions}\label{sec:example}


We give here two examples illustrating that the results of the paper are in a sense optimal. For simplicity, we take the unperturbed operator~$A$ to be defined in the Hilbert space~$L_2(0,2\pi)$ via
	\[
	A = \frac1{i}\frac{d}{dx}
	\]
	subject to the periodic boundary condition $y(0)=y(2\pi)$. The spectrum of $A$ coincides with the set $\mathbb{Z}$, and a normalized eigenfunction $v_n$ corresponding to the eigenvalue~$\lambda_n:=n$ is equal to $e^{inx}/\sqrt{2\pi}$. Therefore, the characteristic function of a generic rank-one perturbation~$B$ of~\eqref{eq:B} has the form 
	\[
		F(z) = \sum_{n\in\mathbb{Z}} \frac{c_n}{n-z} +1,
	\]
	where $c_n : = \overline{a}_nb_n$ is determined via the Fourier coefficients $a_n$ and $b_n$ of the functions $\varphi$ and $\psi$.
	
\begin{example}\rm
	Our first example shows that convergence of the series~\eqref{eq:mu-converge} is not guaranteed if the functions~$\varphi$ and $\psi$ do not belong to $L_2(0,2\pi)$. Namely, we take $a_n = a_{-n} = n^{-1/2}$ and $b_n = -b_{-n} = n^{-1/2}$ for $n\in\mathbb{N}$ and $a_0 = b_0 = 0$; thus $c_n = n^{-1}$ for $n\ne0$. To study the asymptotics of the corresponding eigenvalues $\mu_n$ of the operator~$B$, we recall the equality~\cite[Ch.~5.2]{Ahl:78}
	\[
		\sum_ {\substack{n \in \mathbb{Z}\\ n \neq 0}} \frac{1}{n(n-z)} = \frac1{z^2} -\frac\pi{z} \cot\pi z; 
	\]
	thus 
	\[
		F(z) = \sum_{\substack{n \in \mathbb{Z}\\ n \neq 0}} \frac{1}{n(n-z)} + 1 = \frac{z^2 +1}{z^2} -\frac\pi{z} \cot\pi z.
	\]
	It follows that $\mu_n$ are zeros of the equation
	\[
		\tan \pi z = \frac{\pi z}{z^2 + 1}
	\]
	and thus $\mu_n = \lambda_n + \varepsilon_n$ with $\varepsilon_n \to 0$ as $|n| \to \infty$; the relation
	\[
		\frac{\mu_n}{\varepsilon_n (\mu_n^2 +1)} = \frac{\tan \pi \varepsilon_n}{\pi \varepsilon_n} \to 1
	\]
	as $|n|\to\infty$ now implies that $\varepsilon_n \mu_n \to 1$, and thus $\varepsilon_n = n^{-1}(1 + o(1))$ as $|n|\to\infty$. As a result, the series~\eqref{eq:mu-converge} diverges. 
\end{example}

\begin{example}\label{ex:sin}\rm 
	Consider the rank-one perturbation $B$ of $A$ as in~\eqref{eq:B} with $\varphi$ and $\psi$ given by their Fourier coefficients $a_0 = b_0 = 0$ and 
	$a_n = a_{-n} = n^{-\beta}$ and $b_n = b_{-n} = n^{-\beta}$ for $n\in\mathbb{N}$, with $\beta>1$.
	We observe that the functions $\varphi$ and $\psi$ can be found explicitly via the fractional derivatives, cf. \cite{Tseng:2000}. 	
	The corresponding characteristic function~$F$ is equal to 
	\[
		F(z) = 1 + \sum_{n\ne0}\frac{|n|^{-\beta}}{n-z}, 
	\]
	and can be also represented as a product
	\[
		F(z) = \prod_{\substack{n \in \mathbb{Z} \\ n \neq 0}}\frac{\mu_n - z}{n - z}.
	\]
	The proof of Theorem~\ref{thm:EV-asympt} (see~\eqref{eq:19}) shows that $\mu_n - n = O(|n|^{-\beta})$ as $|n| \to \infty$. 
	The residue of~$F$ at the point $z=n$ is equal to $-|n|^{-\beta}$; on the other hand, it can be calculated as (cf.~\eqref{eq:res-t-F})
	\[
		\res_{z=n} F(z) = (n- \mu_n) \prod\limits_{\substack{m \in \mathbb{Z} \\ m \neq n}} 
		\frac{\mu_m - n}{m - n}.
	\]
	The infinite products in the above formula have been shown in the proof of Lemma~\ref{lem:res-sum} to be uniformly bounded in $n$ (cf.\ the reasoning following formula \eqref{eq:res-sum-3}). Therefore, we conclude that 
	\[
		|n|^{-\beta} \le C |\mu_n - n|,
	\]
	for a constant $C$ independent of~$n$, so that 
	\[
		|\mu_n - n| \asymp |n|^{-\beta}. 
	\]
\end{example}

\begin{remark}
	The same arguments lead to conclusion that, for a generic rank-one perturbation, $|\mu_n - \lambda_n| \asymp |a_nb_n|$ as $|n|\to \infty$. This allows us to control the decay of the offsets $|\mu_n - \lambda_n|$ through the products $|a_nb_n|$ of the Fourier coefficients of $\varphi$ and $\psi$ and vice versa. 
\end{remark}	

Fix an arbitrary function~$\varphi \in H$ and let $a_n$ be its Fourier coefficients in the orthonormal basis $(v_n)_{n\in I}$ of the eigenfunctions~$v_n$ of $A$. Denote by $\ell_1(\varphi)$ the subspace of $\ell_1(I)$ consisting of all sequences $\mathbf{c} = (c_n)_{n\in I}$ of the form $c_n=x_na_n$ with $(x_n)_{n\in I} \in \ell_2(I)$. The above analysis lead to the following uniqueness result:

\begin{corollary}\label{cor:uniq}
Given $\varphi \in H$, for every sequence $\bm{\varepsilon} = (\varepsilon_n)_{n\in I}\in \ell_1(\varphi)$ there exists a function~$\psi \in H$ such that the rank-one perturbation $B$ of the operator~$A$ given by~\eqref{eq:B} has eigenvalues $\mu_n := \lambda_n + \varepsilon_n$, $n\in I$, counting with multiplicities. 

Such $\psi$ is unique if and only if none of $a_n$ vanishes; each $n\in I$ such that $a_n = 0$ leaves the corresponding Fourier coefficient~$b_n$ undetermined and thus increases by one the degree of freedom of the set of all such~$\psi$.

The roles of $\varphi$ and $\psi$ can be interchanged. 
\end{corollary}

\medskip 

We conclude the paper with some comments on the results obtained. 
Most of the analysis of~\cite{DobHry20} and of this paper has straightforward generalization to the case of a normal operator~$A$. The most crucial properties and facts used are 
\begin{itemize}
	\item[(a)] the spectrum of $A$ is simple and separated;
	\item[(b)] the eigenvectors form an orthonormal basis (or even a Riesz basis) of~$H$; 
	\item[(c)] the spectral theorem allowing to represent the characteristic function~$F$ of a rank-one perturbation~$B$ in the form~\eqref{eq:F-new}.
\end{itemize}
Some care should be given to properly choose the regions~$Q_k$ in Section~\ref{sec:asymptotics} and \ref{sec:inverse}, but otherwise the arguments remain valid and establish Theorems~\ref{thm:EV-asympt} and \ref{thm:inv}, i.e., justify the possibility to enumerate the spectrum of~$B$ so that series~\eqref{eq:mu-converge} converges and, for every sequence $(\nu_n)_{n\in I}$ satisfying~\eqref{eq:nu-converge}, to construct a rank-one perturbation $B$ of~$A$ whose spectrum is given by that sequence counting with multiplicities.

In the special case of a self-adjoint rank-one perturbation~\eqref{eq:B} with $\psi = \alpha \varphi$ and $\alpha \in \mathbb{R}$, the resulting spectrum of $B$ is simple outside~$\sigma_0(A)$, of geometric multiplicity at most~$2$ at the points of~$\sigma_0(A)$, and the eigenvalues $\sigma_1(A)$ and $\sigma_1(B)$ strictly interlace, i.e., between every two consecutive eigenvalues from $\sigma_1(A)$ there is a unique eigenvalue from $\sigma_1(B)$ and, vice versa, between every two consecutive eigenvalues from $\sigma_1(B)$ there is a unique eigenvalue from $\sigma_1(A)$. This interlacing property follows from the minmax principle~\cite{ReeSim}; moreover, for $\alpha>0$ we have $\lambda_n < \mu_n$ for all $n\in I_1$; the signs are reversed if~$\alpha<0$. Given $(\mu_n)_{n\in I}$ satisfying the interlacing property, convergence of the series~\eqref{eq:mu-converge} is a necessary and sufficient condition on the spectrum of a self-adjoint rank-one perturbation~$B$ of $A$.

\end{document}